\documentclass[12pt]{amsart}
\usepackage{latexsym}
\usepackage{amsmath}
\usepackage{amsfonts}
\usepackage{amsthm}
\usepackage{amssymb}
\usepackage{ifthen}
\usepackage{enumerate}
\usepackage[T1]{fontenc}
\def\Eq#1#2{\ifthenelse{\equal{#1}{*}}
  {\begin{equation*}\begin{aligned}#2\end{aligned}\end{equation*}}
  {\begin{equation}\begin{aligned}\label{E#1}#2\end{aligned}\end{equation}}}

\newtheorem{thm}{Theorem}
\newtheorem{prop}{Proposition}
\newtheorem{coro}{Corollary}
\newtheorem{lemm}{Lemma}

\theoremstyle{remark}
\newtheorem{remark}{Remark}
\newtheorem{exmp}{Example}
\newtheorem{claim}{Claim} 
\newtheorem{prclm}{Proof of Claim}[section]

\theoremstyle{definition}
\newtheorem{defin}{Definition}
\newtheorem{probl}{Problem}
\newtheorem{opprob}{Open Problem}

\newcommand{\NN}{\mathbb{N}}

\newcommand{\RR}{\mathbb{R}}

\newcommand{\Rn}{\mathbb{R}^n}

\newcommand{\exmo}{extended monotone left-inverse }

\newcommand{\nma}{\Vert}

\newcommand{\conv}{\mbox{\rm conv}}

\newcommand{\map}{\longrightarrow}

\newcommand{\bsz}[2]{\left \langle #1 \,, #2 \right \rangle}

\author[P\'eter T\'oth]{P\'eter T\'oth}

\title[Generalized inverses]{Generalized inverses of strictly 
monotone transformations}

\address{Institute of Mathematics,
University of Debrecen,
4002 Debrecen, Pf.~400, Hungary}
\email{toth.peter@science.unideb.hu}



\thanks{
The research 
has been supported by 
the EKÖP-25-0 University Research Scholarship Program of 
the Ministry for Culture and Innovation 
from the source of the 
National Research, Development and Innovation Fund, 
by the PhD Excellence Scholarship from the 
Count Istv\'an Tisza Foundation 
for the University of Debrecen and by 
the University of Debrecen Program 
for Scientific Publication.}

\begin{document}

\begin{abstract} 
Let $ C \subseteq \Rn $ be a convex set. The mapping 
$ f : C \longrightarrow \Rn $ is strictly increasing, if 
$ \langle f(x)-f(y), x-y \rangle > 0 $ 
for all distinct elements 
$ x,y \in C $. Applying classical theorems of 
finite dimensional convex geometry and convex analysis, 
we show that the inverse functions has a 
unique extension 
$ f^{(-1)} : \conv (f(C)) \longrightarrow C $ 
such that $ f^{(-1)} $ is monotone, continuous 
and it acts as a left-inverse of $ f $. 
As an application we introduce the concept of 
vector-valued weighted quasi-arithmetic means and discuss 
their equality problem. 
\end{abstract}

\maketitle


\section{Introduction}

It is an elementary fact that if $ I \subseteq \RR $ is an 
interval and $ f : I \map \RR $ is a continuous, strictly 
monotone function, then $ f(I) $ is again an interval, and 
the inverse function $ f^{-1} : f(I) \map I $ is also 
strictly monotone (in the same sense) and continuous. 
When $ f $ is not continuous, the ordinary 
inverse $ f^{-1} $ is still strictly monotone, 
but its domain of definition is not an interval. 
However, it is known 
that $ f^{-1} $ can be extended to the convex hull of 
$ f(I) $ such that this extension is still monotone, 
what is more, continuous. This result is proved explicitly 
in the paper \cite{GP20} by Grünwald and Páles and 
it also appears in the textbook \cite{GaPa16}. 
A famous particular case is the 
{\em quantile function} of a strictly increasing CDF 
of a random variable. 

\begin{prop}{\cite[Lemma 1]{GP20}}
Let $ I \subseteq \RR $ be an interval and let 
$ f : I \map \RR $ be a 
strictly monotone function. Then there exists 
a uniquely determined monotone function 
$ f^{(-1)} : \conv \left( f(I) \right) \map I $ 
such that 
\[
f^{(-1)} \left( f(x) \right) = x 
\hspace{10mm} \left( \, x \in I \, \right).  
\]
Moreover, $ f^{(-1)} $ is continuous. 
\end{prop}

During the 
{\em 60th International Symposium on Functional Equations} 
Zs. Páles proposed an Open Problem (see \cite[4. Problem]{60ISFE}) 
about the generalization of the previous observation to 
higher dimensional spaces. 
Before formulating the precise question we clarify some 
notions and notations. Throughout the whole paper 
$ n $ denotes a positive integer. We consider the $ n $ 
dimensional euclidean space $ \Rn $ equipped with 
the standard inner product, 
the induced norm and topology.  The inner product of 
$ x,y \in \Rn $ will be denoted by $ \langle x, y \rangle $. 
For an arbitrary nonempty set $ S \subseteq \Rn $ its convex 
hull is denoted by $ \conv(S) $. 
For a detailed introduction to fundamental concepts of 
convex geometry we recommend the monographs 
of Lay \cite{Lay82} and Rockafellar \cite{Roc70}. 

\begin{defin}\label{def-StrictlyMon}
Let $ K \subseteq \Rn $ be a nonempty, convex set. 
We say that $ f : K \map \Rn $ 
is an {\em increasing mapping}, if 
\[
\left \langle \, f(x) - f(y) \, , \, 
x - y \, \right \rangle \geq 0 
\] 
for all $ x,y \in K $. Similarly, $ f $ is called a 
{\em decreasing mapping} if 
\[
\left \langle \, f(x) - f(y) \, , \, 
x - y \, \right \rangle \leq 0 
\] 
for all $ x,y \in K $. 
Moreover, if the above inequalities are strict 
for all elements $ x,y \in K $ 
such that $ x \neq y $, then we call these 
{\em strictly increasing/strictly decreasing mapping}s, 
respectively.  
\end{defin}

A notable example for increasing mappings is the gradient 
map of a convex function of $ n $ variables. 
In the $ n=1 $ dimensional case the notion is 
equivalent to the standard (strict) monotonicity 
of a real function $ f $ defined on an interval, 
as the inner product 
is the ordinary multiplication. 
Notice that multiplying an increasing mapping by $ (-1) $ 
we obtain a decreasing one. Thus, in the sequel we mainly 
focus on (strictly) increasing mappings. 
It is easy to see that 
if $ f $ is strictly increasing, then it has to be injective. 

Consequently, the ordinary inverse $ f^{-1} $ of a strictly 
increasing mapping always exists. The main question of 
our paper (based on the open problem by Páles) is the 
following: {\em Does there exist an extension of $ f^{-1} $ 
to the convex hull of the image, which acts as a 
left-inverse and remains monotone? 
Does such an extended inverse have some additional 
properties such as continuity?} More precisely, 
we consider the problem below.  

\begin{probl}\label{prob-MainPR}{\cite[4. Problem by Zs. Páles]{60ISFE}}
Let $ \emptyset \neq K \subseteq \Rn $ be a convex set and 
let $ f : K \map \Rn $ be a strictly increasing mapping. 
Does there exist an increasing mapping 
$ f^{(-1)} : \conv \left( f(K) \right) \map K  $ such that 
$ f^{(-1)} \left( f(x) \right) = x $ for all $ x \in K $, 
moreover 
\[
\langle \, f^{(-1)}(x) - f^{(-1)}(y) 
\, , \, 
x - y \, \rangle \geq 0 
\]
for all $ x,y \in \mbox{conv} \left( f(K) \right) $? 
Can $ f^{(-1)} $ be continuous? 
\end{probl}

We give a positive answer to these questions 
assuming that the domain is closed and convex. 
We demonstrate that the additional topological property 
of the domain is crucial, which is in contrast to the 
$1$-dimensional setting. 
A function $ f^{(-1)} $ fulfilling the desired properties 
is going to be called the {\em extended monotone 
left-inverse of $ f $}. When it is not confusing, 
we simply refer to this function 
as the {\em generalized inverse of $ f $}. 

In Section \ref{Section-Appl}, as 
an application, we introduce the concept of vector-valued 
weighted quasi-arithmetic means defined on a closed, convex 
subset of $ \Rn $. This is the natural generalization of 
the classical notion: the $ m $-variable quasi-arithmetic 
mean (briefly, QAM) 
generated by the strictly increasing, continuous 
function $ \varphi : I \map I $ is 
\[
\varphi^{-1} \left( 
\frac{\varphi(x_1) + \dots + \varphi(x_m)}{m} \right)
\hspace{10mm} \left( x_1 \,, \dots , x_m \in I \right), 
\]
where $ I \subseteq \RR $ is an interval. 
Different characterizations of the QAM are due to 
Kolmogorov \cite{Kol30}, Nagumo \cite{Nag30}, 
De Finetti \cite{dFi31}, Aczél \cite{Acz66}, 
and recently Burai--Kiss--Szokol \cite{BKSz21, BKSz23}. 
We discuss the equality problem of 
weighted vector-valued QAMs 
for a fixed number of variables. 
The solution in the real valued setting is elaborated in 
the classical monograph of Hardy--Littlewood--Pólya 
\cite{HLP34}. 

Recent results about the equality problem 
of several generalizations of QAMs are contained in 
\cite{Bur19, GP20, KP26, LPZ20, PP26, PZ22}. 
These articles concern real variables. 
For concepts of vector valued means we 
shall mention the conference paper \cite{Nie23} of Nielsen 
following a differential geometric approach and the 
paper of Leonetti \cite{Leo23}. 

\section{Main results} 

Our first main objective is to guarantee the 
existence of a generalized left-inverse for strictly 
increasing mappings defined on closed, convex domains. 
In order to construct the generalized inverse, we begin 
with establishing necessary conditions for the existence. 

\begin{prop}\label{prop-NecessCond}
Let $ \emptyset \neq K \subseteq \Rn $ 
be a convex set, $ f : K \map \Rn $ 
be a strictly increasing mapping, and assume that 
$ f^{(-1)} : \conv \left( f(K) \right) \map K $ 
is an \exmo of $ f $. 
Then 
\begin{equation}\label{eq-NcCond1}
f^{(-1)}(z) \in 
\bigcap_{x \in K} 
\left\lbrace s \in K : 
\left\langle x - s \,,\, 
f(x)- z \right\rangle \geq 0 \right\rbrace  
\end{equation} 
for all $ z \in \conv \left( f(K) \right) $. 
\end{prop}

\begin{proof}
Since $ f^{(-1)} $ is an increasing mapping, the inequality 
\[
\langle  f^{(-1)}(w) -  f^{(-1)}(z) , w-z \rangle \geq 0 
\]
holds for all vectors 
$ w,z \in \conv \left( f(K) \right) $. 
In particular, if $ w \in f(K) $, that is, there exists 
$ x \in K $ such that $ w = f(x) $, then 
$
\langle  f^{(-1)}(f(x)) -  f^{(-1)}(z) , f(x)-z \rangle \geq 0. 
$
Due to the left-inverse property, this means that 
\[
\langle  x -  f^{(-1)}(z) , f(x)-z \rangle \geq 0 
\] 
for any $ x \in K $ and $ z \in \conv (f(K))$. Therefore 
$ f^{(-1)}(z) $ is indeed contained in 
the intersection. 
\end{proof}

Let us observe that in the intersection \eqref{eq-NcCond1} 
each member is convex set. Indeed, every set is obtained 
as the intersection of the convex set $ K $ and the 
solution set of the inner product inequality, which is a 
closed half-space by definition. 
A powerful classical result about the intersection of 
convex sets of $ \Rn $ is Helly's Theorem. The theorem 
has numerous variations, thus we explicitly formulate 
the most suitable one. 

\begin{thm}[Helly's Theorem]\label{thm-Helly} 
Let $ \mathcal{C} $ be a family of nonempty 
closed convex sets in $ \Rn $ 
and suppose that 
$ \mathcal{C} $ contains at least 
$ n+1 $ members. 
If there exists a finite subfamily 
$ \mathcal{C}_0 \subseteq \mathcal{B} $ such that 
$ \bigcap \mathcal{C}_0 $ is bounded, moreover 
every subfamily of $ n+1 $ sets 
from $ \mathcal{C} $ has a nonempty intersection, 
then $ \bigcap \mathcal{C} \neq \emptyset $. 
\end{thm}

A slightly stronger version can be found in 
\cite[Corollary 21.3.2]{Roc70}. The following Lemma 
of pure linear algebra 
will be important to ensure 
that the assumptions of Helly's Theorem hold.  

\begin{lemm}\label{lemma-LinAlg1}
Let $ n,d \in \NN $ and let us assume that for the vectors 
$ x_1 \,, \dots , x_d \in \Rn $ and 
$ y_1 \,, \dots , y_d \in \Rn $ the inequalities 
\begin{equation}\label{eq1-lemma-linalg}
\bsz{x_i - x_j}{y_i - y_j} \geq 0 
\hspace{10mm}  
\left( i,j = 1 , \dots , d \right)
\end{equation} 
are fulfilled. Then, for any $ z \in \Rn $, 
there exists $ v \in \conv \left( 
x_1 \,, \dots, x_d \right) $ such that the 
inequalities 
\begin{equation}\label{eq2-lemma-linalg}
\bsz{x_j - v}{y_j - z} \geq 0 
\hspace{10mm} 
\left( j = 1 , \dots , d \right) 
\end{equation} 
hold. 
\end{lemm}

\begin{proof}
Let us introduce the following two notations: 
\[
C:= \conv \left( x_1 \,, \dots , x_d \right) 
\hspace{5mm} \mbox{ and } \hspace{5mm} 
\Delta_d := 
\lbrace a = 
\left( a_1 \,, \dots \,a_d \right) \in \left[ 0,1 \right]^d 
\, : \, a_1 + \dots + a_n = 1 \rbrace. 
\] 
We have to show that there exists $ v \in C $ 
such that 
\[
\min_{k=1, \dots , d} 
\bsz{x_k-v}{y_k-z} \geq 0. 
\]
However, this is equivalent to the assertion that 
\[
\max_{v \in C} \min_{k=1, \dots , d} 
\bsz{x_k-v}{y_k-z} \geq 0. 
\]
On the other hand, for any fixed $ v \in C $, 
$ \min_{k=1, \dots , d} \bsz{x_k-v}{y_k-z} $ is the 
minimum of $ d $ real numbers, so it is the minimum 
of the convex hull of these numbers. That is, 
\[
\min_{k=1, \dots , d} \bsz{x_k-v}{y_k-z} = 
\min_{a \in \Delta_d} 
\sum_{k=1}^d a_k \bsz{x_k-v}{y_k-z}. 
\]
Consequently, we need to verify that 
\begin{equation}\label{eq-LinAlglem-minmax1} 
\max_{v \in C} \min_{a \in \Delta_d} 
\sum_{k=1}^d a_k \bsz{x_k-v}{y_k-z} \geq 0. 
\end{equation}
Let us define the function $ G : C \times \Delta_d 
\map \RR $ with the following formula: 
\[
G(v,a) := \sum_{k=1}^d a_k \bsz{x_k-v}{y_k-z} 
\hspace{10mm} \left( v \in C, a \in \Delta_d \right).  
\]
It is clear that $ G $ is defined on the Cartesian product 
of two compact, convex subsets of $ \RR^d $, moreover 
$ G $ is affine in its first variable and 
linear in its second variable. Therefore, 
Ky Fan's Minimax Theorem \cite{Fan53} is applicable in 
\eqref{eq-LinAlglem-minmax1}, so we have to 
show 
\begin{equation}\label{eq-LinAlglem-minmax2} 
\min_{a \in \Delta_d} \max_{v \in C} 
G(v,a) 
= 
\min_{a \in \Delta_d} \max_{v \in C} 
\sum_{k=1}^d a_k \bsz{x_k-v}{y_k-z} \geq 0. 
\end{equation}

For any fixed vector 
$ a = \left( a_1 \,, \dots a_d \right) \in \Delta_d $ 
the convex combination 
$ v_a = a_1 x_1 + \dots + a_d x_d $ is contained in $ C $. 
Now we may calculate 
\begin{align}
G(v_a,a) & = 
\sum_{k=1}^d a_k \bsz{x_k-v_a}{y_k-z} = 
\sum_{k=1}^d a_k \bsz{x_k- \sum_{j=1}^d a_j x_j}{y_k-z} \\ 
\label{hehe}
& = \sum_{k=1}^d \sum_{j=1}^d a_k a_j \bsz{x_k-x_j}{y_k} 
- \sum_{k=1}^d \sum_{j=1}^d a_k a_j \bsz{x_k-x_j}{z}, 
\end{align}
using the fact that $ \sum_{j=1}^d a_j = 1 $. 
The second term is clearly zero, since 
\[
2 \cdot \sum_{k=1}^d \sum_{j=1}^d a_k a_j \bsz{x_k-x_j}{z} 
= \sum_{k=1}^d \sum_{j=1}^d a_k a_j 
\bigl( \bsz{x_k-x_j}{z} + \bsz{x_j-x_k}{z} \bigr) = 0.  
\]
For the investigation of the first term we shall utilize 
assumption \eqref{eq1-lemma-linalg}, whence 
\begin{align*}
& 0 \leq 
\sum_{k=1}^d \sum_{j=1}^d a_k a_j \bsz{x_k-x_j}{y_k-y_j} \\ 
& = \sum_{k=1}^d \sum_{j=1}^d a_k a_j \bsz{x_k-x_j}{y_k} 
- \sum_{k=1}^d \sum_{j=1}^d a_k a_j \bsz{x_k-x_j}{y_j} \\ 
& = \sum_{k=1}^d \sum_{j=1}^d a_k a_j \bsz{x_k-x_j}{y_k} 
+ \sum_{j=1}^d \sum_{k=1}^d a_j a_k \bsz{x_j-x_k}{y_j} 
= 2 \cdot \sum_{k=1}^d \sum_{j=1}^d a_k a_j \bsz{x_k-x_j}{y_k},  
\end{align*} 
so the first term in \eqref{hehe} is non-negative. 
Consequently, we have obtained that for any $ a \in \Delta_d $ 
there exists $ v_a \in C $ such that $ G(v_a, a) \geq 0 $. 
Thus \eqref{eq-LinAlglem-minmax2} is indeed fulfilled, 
which is equivalent to \eqref{eq-LinAlglem-minmax1}, 
and that completes the proof. 
\end{proof}

\begin{lemm}\label{lemm-Intersection_nmpty}
Let $ \emptyset \neq C \subseteq \Rn $ be a closed convex set 
and let $ f : C \map \Rn $ be a strictly increasing mapping. 
Suppose that, for any $ z \in \conv \left( f(C) \right) $ 
and for any $ x_1 \,, \dots , x_n\,, x_{n+1} \in C $, 
we have 
\begin{equation}\label{eq-prop_HellyCond} 
\bigcap_{j=1}^{n+1} 
\lbrace s \in C \, : \, 
\bsz{x_j-s}{f(x_j)-z} \geq 0 
\rbrace \neq \emptyset. 
\end{equation}
Then 
\[
\bigcap_{x \in C} 
\lbrace s \in C \, : \, 
\bsz{x-s}{f(x)-z} \geq 0 
\rbrace \neq \emptyset. 
\]
\end{lemm}

\begin{proof}
The case when $ C $ is a singleton is trivial, so from 
now on suppose that this is not the case.  
The idea is to use Helly's Theorem. 
Let us consider the set 
\[
C_x := 
\lbrace s \in C \, : \, 
\bsz{x-s}{f(x)-z} \geq 0 
\rbrace 
\]
for every $ x \in C $. It is obvious that $ x \in C_x \,$, 
moreover $ C_x $ is the intersection of a half-space and 
$ C $, therefore it is closed and convex. That is, 
$ \mathcal{C} := \lbrace C_x \, : \, x \in C \rbrace $ 
is an infinite family of nonempty closed convex subsets 
of $ \Rn $. 
The assumption \eqref{eq-prop_HellyCond} ensures that 
every subfamily of $ n+1 $ sets from $ \mathcal{C} $ 
has nonempty intersection. The last assumption of 
Helly's Theorem which needs to be checked is the 
existence of a finite subfamily of $ \mathcal{C} $ with 
a bounded intersection. 

If $ C $ itself is bounded, this is evident. 
Suppose the contrary. Then the recession cone 
of $ C $ is non-trivial (see \cite[Theorem 8.4]{Roc70}). 
That is, 
\[
R := \lbrace h \in \Rn \, : \, x+h \in C 
\mbox{ for all } x \in C \rbrace \neq \lbrace 0 \rbrace. 
\] 
It is well-known (see \cite[Theorem 8.2]{Roc70}) 
that $ R $ is a 
closed, convex cone, since $ C $ is closed. 
We shall also consider 
\[
R_0 := \left \lbrace \frac{h}{\nma h \nma} \, : \, 
h \in R \setminus \{ 0 \} \right \rbrace, 
\]
which is the intersection of 
$ R $ and the unit sphere $ \mathcal{S}_{n-1} \,$, 
so $ R_0 $ is compact. 
\begin{claim}\label{clm-prop_Claim1}
For any $ h \in R_0 $ there exists $ x_h \in C $ 
such that 
$ \bsz{h}{f(x_h)-z} > 0 $. 
\end{claim}

\begin{prclm}
Fix an arbitrary $ h \in R_0 $. In the first step 
we will show that there is a vector $ w \in C $ fulfilling 
$ \bsz{h}{f(w)-z} \geq 0 $. Since $ z \in \conv (f(C)) $, 
there exist 
\[
\lambda_1 \,, \dots , \lambda_n \,, \lambda_{n+1} 
\in \left[ 0,1 \right] 
\hspace{3mm} \mbox{ and } \hspace{3mm} 
y_1 \,, \dots , y_n \,, y_{n+1} \in C 
\]
such that 
$ z = \lambda_1 f(y_1) + \dots + \lambda_n f(y_n) 
+ \lambda_{n+1} f(y_{n+1}) $ and 
$ \sum_{i=1}^{n+1} \lambda_i = 1 $, according to 
Carathéodory's Theorem. By calculating 
\begin{align*}
0 & = \bsz{h}{0} = \bsz{h}{z-z} 
= \bsz{h}{\sum_{i=1}^{n+1} \lambda_i f(y_i) 
- \sum_{i=1}^{n+1} \lambda_i z} \\ 
& = \sum_{i=1}^{n+1} \bsz{h}{\lambda_i \left( f(y_i)-z \right)} 
= \sum_{i=1}^{n+1} \lambda_i \bsz{h}{f(y_i)-z}, 
\end{align*} 
we obtain that $ \bsz{h}{ f(y_{i_0})-z } \geq 0 $ 
for at least one index $ i_0 \,$. Thus for $ w := y_{i_0} $ 
we indeed have $ \bsz{h}{f(w)-z} \geq 0 $. 
Since $ h \in R $, we know that $ x_h := w+h \in C $ as well. 
Now 
\begin{align*}
\bsz{h}{f(x_h) - z} & = 
\bsz{h}{f(x_h)-f(w)+f(w)-z} = 
\bsz{h}{f(x_h)-f(w)} \\ 
& +  
\bsz{h}{f(w)-z} 
= \bsz{x_h-w}{f(x_h)-f(w)} + \bsz{h}{f(w)-z} > 0, 
\end{align*} 
as the first term is positive (because $ f $ 
is strictly increasing) while the second term is non-negative 
due to the previous step. 
\flushright 
$ \boxtimes $ 
\end{prclm}
Now we are able to construct an open cover for $ R_0 \,$. 
For each $ x \in C $ let us define the open 
half-space 
\[ 
H_x := \left \lbrace 
v \in \Rn \, : \, \bsz{v}{f(x)-z} > 0. 
\right \rbrace
\]
According to Claim \ref{clm-prop_Claim1}, 
every vector $ h \in R_0 $ is contained in the corresponding 
set $ H_{x_h} \,$. Hence 
$ R_0 \subseteq \bigcap_{x \in C} H_x \,$. 
But $ R_0 $ is compact, so there exists a finite subcover, 
i.~e. there exist $ m \in  \NN $ and 
$ u_1 \,, \dots , u_m \in C $ such that 
\[
R_0 \subseteq \bigcap_{k = 1}^{m} H_{u_k} \,. 
\]
\begin{claim}\label{clm-prop_Claim2}
$ M:= \bigcap_{k=1}^{m} 
\lbrace s \in C \, : \, 
\bsz{u_k-s}{f(u_k)-z} \geq 0 
\rbrace $ 
is a nonempty compact, convex set. 
\end{claim}
\begin{prclm}
The classical Helly's Theorem for finite families 
of convex sets provides that $ M \neq \emptyset $, 
using assumption \eqref{eq-prop_HellyCond}. 
Closedness and convexity is obvious, so we only need 
to verify that $ M $ is bounded. Let $ Q $ denote its 
recession cone. Since $ M \subseteq C $, we know that 
$ Q \subseteq R $. 
Suppose that $ Q $ is nontrivial, so there exists 
$ 0 \neq q \in Q $ and hence 
\[ 
q_0 := \frac{q}{\nma q \nma} \in R_0 \,. 
\] 
But then there exists 
$ \ell \in \{ 1, \dots , m \} $ such that 
$ q_0 \in H_{u_\ell} \,$. This means 
$ \bsz{q_0}{f(u_{\ell}) - z} > 0 $. 
If we pick an arbitrary $ m \in M $, then 
$ m + \mu \cdot q \in M $ for any $ \mu > 0 $, because 
$ q \in Q $. This implies 
\begin{align*}
0  \leq \bsz{u_{\ell}- \left( m + \mu q \right)}{f(u_{\ell}) - z} 
= \bsz{u_{\ell}-m}{f(u_{\ell}) - z} - 
\mu \cdot \bsz{q}{f(u_{\ell}) - z} 
\end{align*} 
for every $ \mu > 0 $. 
However, we shall observe that, 
due to $ \bsz{q_0}{f(u_{\ell} - z)} > 0 $, 
the right hand side tends to $ - \infty $ 
as $ \mu \to + \infty $. This contradiction means that 
$ Q = \{ 0 \} $, so $ M $ has a trivial recession cone, 
hence it is bounded. 
\flushright 
$ \boxtimes $ 
\end{prclm}
According to Claim \ref{clm-prop_Claim2}, $ \mathcal{C} $ 
indeed has a finite subfamily with bounded intersection, 
namely $ \bigcap_{k=1}^{m} C_{u_k} = M $ is a compact set. 
Eventually we are able to apply Theorem \ref{thm-Helly} for 
$ \mathcal{C} $ and obtain that 
$ \bigcap \mathcal{C} \neq \emptyset $ 
which had to be verified. 
\end{proof}

\begin{remark}\label{rem-Lemma_Compact}
Let us note that when $ C $ is 
compact, then each set in the family 
$ \mathcal{C} $ is bounded even if 
$ z $ is not contained in $ \conv (f(C)) $. 
Thus if we suppose that $ C $ is compact, convex, then 
the statement of Lemma \ref{lemm-Intersection_nmpty} 
remains valid 
for arbitrary $ z \in \Rn $, while the proof is just 
a direct application of Helly's Theorem. 
\end{remark}

Combining the previous two Lemmas 
we are able to prove our main result about 
the existence and uniqueness of a generalized inverse. 

\begin{thm}\label{thm-MainThm-closed}
Let $ \emptyset \neq C \subseteq \Rn $ be a 
closed convex set and let 
$ f : C \map \Rn $ be a strictly increasing mapping. 
Then there exists a function 
$ f^{(-1)} : \conv \left( f(C) \right) \map C $ 
such that 
\begin{itemize}
\item $ f^{(-1)} \left( f(x) \right) = x $ 
for every $ x \in C $, 
\item $ \bsz{ f^{(-1)}(u)- 
f^{(-1)}(v) }{ u-v } \geq 0 $  
for all $ u,v \in \conv \left( f(C) \right) $. 
\end{itemize}
Moreover, $ f^{(-1)} $ is uniquely determined. 
\end{thm}

\begin{proof}
Let $ z \in \conv (f(C)) $ and 
$ x_1 \,, \dots , x_n \,, x_{n+1} \in C $ 
be arbitrary points. Using the notation 
$ y_j := f(x_j) $ (for $ j = 1 , \dots , n+1 $), 
the strictly increasing property of $ f $ yields 
\[
\bsz{x_i-x_j}{y_i-y_j} \geq 0 
\hspace{10mm}
\left( i,j = 1 , \dots , d \right). 
\]
This means that the assumption \eqref{eq1-lemma-linalg} 
in Lemma \ref{lemma-LinAlg1} is fulfilled. Therefore 
there exists 
$ v \in \conv \left( x_1 \,, \dots , x_{n+1} \right) 
\subseteq C $ such that 
\[
\bsz{x_j - v}{y_j - z} \geq 0 
\hspace{10mm} 
\left( j = 1 , \dots , d \right). 
\]
That is, 
\[
\bigcap_{j=1}^{n+1} 
\lbrace s \in C \, : \, 
\bsz{x_j-s}{f(x_j)-z} \geq 0 
\rbrace \neq \emptyset. 
\]
Now the assertion of 
Lemma \ref{lemm-Intersection_nmpty} implies that 
\[
M_z := 
\bigcap_{x \in C} 
\lbrace s \in C \, : \, 
\bsz{x-s}{f(x)-z} \geq 0 
\rbrace \neq \emptyset. 
\]
We are going to show that $ M_z $ is actually a singleton. 
Suppose that $ v,w \in M_z $ and consider the points 
$ t_1 := \frac{v+3w}{4} $ and $ t_2 := \frac{3v+w}{4} $. 
Now $ t_1 \,, t_2 \in C $, moreover 
$ t_1 - t_2 = \frac{w-v}{2} $. Since $ v,w \in M_z $, 
the inequalities 
\[
\bsz{t_1-w}{f(t_1)-z} \geq 0 
\hspace{3mm} \mbox{ and } \hspace{3mm} 
\bsz{t_2-v}{f(t_2)-z} \geq 0 
\]
are fulfilled. But these are equivalent to 
\[
\bsz{\frac{v-w}{4}}{f(t_1)-z} \geq 0 
\hspace{3mm} \mbox{ and } \hspace{3mm} 
\bsz{\frac{w-v}{4}}{f(t_2)-z} \geq 0, 
\]
respectively. Multiplying by $ 2 $ and adding them up we 
get 
\[
0 \leq \bsz{\frac{w-v}{2}}{f(t_2) - f(t_1)} = 
\bsz{t_1-t_2}{f(t_2) - f(t_1)}. 
\]
Due to $ f $ being strictly increasing, this can occur 
only when $ t_1 = t_2 $ and, equivalently, $ v=w $. 
Thus $ M_z $ is indeed a singleton. 
The necessary condition in Proposition \ref{prop-NecessCond} 
states that if $ f^{(-1)} $ exists then 
$ f^{(-1)}(z) \in M_z $ must hold on its domain. 
This gives the unique definition of 
$ f^{(-1)} : \conv (f(C)) \map C $, namely 
\[
\mbox{let } f^{(-1)}(z) 
\mbox{ be the unique element of } M_z 
\mbox{ for every } z \in \conv (f(C)). 
\] 
Finally we have to show that this function is indeed a 
left inverse, and it is monotone increasing. 
For any $ x,y \in C $ we have 
$ \bsz{x-y}{f(x)-f(y)} \geq 0 $, so $ x \in M_{f(x)} $. 
Thus $ f^{(-1)}(f(x)) = x $ for all $ x \in C $. 
On the other hand, consider any two vectors 
$ u,v \in \conv (f(C)) $ and introduce the notations 
\[
d := \frac{f^{(-1)}(v) - f^{(-1)}(u)}{2} 
\hspace{3mm} \mbox{ and } \hspace{3mm} 
m := \frac{f^{(-1)}(u) + f^{(-1)}(v)}{2}.
\]
Due to $ m \in C $ and using the definition of 
$ f^{(-1)} $ we have 
\[
\bsz{m-f^{(-1)}(u)}{f(m)-u} \geq 0  
\hspace{3mm} \mbox{ and } \hspace{3mm} 
\bsz{m-f^{(-1)}(v)}{f(m)-v} \geq 0. 
\]
That is, $ \bsz{d}{f(m)-u} \geq 0 $ and 
$ \bsz{-d}{f(m)-v} \geq 0 $. After summation we get 
\[
0 \leq \bsz{d}{v-u} \leq \bsz{2d}{v-u} 
= \bsz{f^{(-1)}(v) - f^{(-1)}(u)}{v-u}. 
\]
Therefore $ f^{(-1)} $ turns out to be increasing, 
thus the proof is complete.  
\end{proof}

\begin{remark}
We wish to emphasize that the proof provides the 
definition of the \exmo explicitly. For 
any $ z \in \conv (f(C)) $ we have 
\[
\lbrace f^{(-1)}(z) \rbrace = 
\bigcap_{x \in C} 
\lbrace s \in C \, : \, 
\bsz{x-s}{f(x)-z} \geq 0 
\rbrace. 
\] 
One should also observe that the only occasion where 
we rely on the fact that $ z $ is contained in 
$ \conv (f(C)) $ is when we apply 
Lemma \ref{lemm-Intersection_nmpty}. 
However, we have already mentioned in 
Remark \ref{rem-Lemma_Compact} that if the 
domain is compact then 
Lemma \ref{lemm-Intersection_nmpty} remains valid 
for any $ z \in \Rn $. 
Thus if the domain of definition of a strictly increasing 
mapping is a compact convex set, then 
the generalized inverse has a unique extension to the 
whole space. We formulate this in the following Corollary. 
\end{remark}

\begin{coro}\label{coro-Compact_domain}
Let $ \emptyset \neq K \subseteq \Rn $ be a 
compact, convex set and let
$ f : K \map \Rn $ be a strictly increasing function. 
Then there exists a function 
$ f^{(-1)} : \Rn \map K $ 
such that 
\begin{itemize}
\item $ f^{(-1)} \left( f(x) \right) = x $ 
for every $ x \in K $, 
\item $ \bsz{ f^{(-1)}(u)- 
f^{(-1)}(v) }{ u-v } \geq 0 $  
for all $ u,v \in \Rn  $. 
\end{itemize}
Moreover, $ f^{(-1)} $ is uniquely determined by the 
identity 
\[
\lbrace f^{(-1)}(z) \rbrace = 
\bigcap_{x \in K} 
\lbrace s \in K \, : \, 
\bsz{x-s}{f(x)-z} \geq 0 
\rbrace
\hspace{10mm} \left( z \in \Rn \right). 
\] 
\end{coro}

\section{Topological properties of $ f^{(-1)} $ and the domain}

In the one-dimensional case the only important property 
of the domain was convexity, the existence and uniqueness 
of the generalized inverse holds for open, closed and 
half-closed intervals. In the following example we 
demonstrate that in higher dimensions the closedness 
of the domain is a crucial condition. The idea is 
based on a counterexample presented by 
K. Okamura during the mentioned {\em 60th ISFE} meeting 
(see \cite[5. Remark]{60ISFE}). 

\begin{exmp}\label{exmpl-Okamura}
Consider the convex open half-disk 
$ S = \lbrace (x,y) \in \RR^2 : x > 0 
\mbox{ and } x^2 + y^2 < 1 \rbrace $ 
and the mapping 
$ f : S \map \RR^2 $ defined by 
$ f(x,y) = \left( x^2 - y^2 , 2xy \right) $. 
Then $ f $ is strictly increasing, but 
it has no \exmo. 
\end{exmp}

\begin{proof}
The fact that $ f $ is strictly increasing is a matter 
of elementary calculations. An elegant reasoning is 
contained in \cite[5. Remark]{60ISFE}, applying an 
identification between $ f $ and the complex square 
function $ z \mapsto z^2 $. 
Since 
\[
S = \left \lbrace (r \cos \varphi, r \sin \varphi) \in \RR^2 
\, : \, 0 < r < 1, 
\mbox{ and } 
-\frac{\pi}{2}  < \varphi < \frac{\pi}{2} \right \rbrace 
\]
and $ f \left( r \cos \varphi, r \sin \varphi \right) 
= \left( r^2 \cos \left( 2 \varphi \right) , 
r^2 \sin \left( 2 \varphi \right) \right) $, it is clear that 
\begin{align*}
f(S) & = 
\left \lbrace 
\left( r^2 \cos \left( 2 \varphi \right) , 
r^2 \sin \left( 2 \varphi \right) \right) 
\, : \, 
0 < r < 1, 
-\frac{\pi}{2}  < \varphi < \frac{\pi}{2} 
\right \rbrace  
\\ & = 
\left \lbrace 
\left( R \cos \psi , 
R \sin \psi \right) 
\, : \, 
0 < R < 1, -\pi  < \psi < \pi 
\right \rbrace 
\\ & = 
B \setminus 
\left \lbrace (x,0) \, : \, x \leq 0 \right \rbrace. 
\end{align*} 
Here 
$ B = \lbrace x \in \RR^2 \, : \, \nma x \nma < 1 \rbrace $ 
is the open unit ball. 
Consequently, $ \conv f(S) = B $. Suppose that there 
exists an \exmo $ F : B \map S $.  
In particular, there should exist $ (a,b) \in S $ such 
that $ F(0,0) = (a,b) $. 
Let us now investigate 
the necessary condition of Proposition \ref{prop-NecessCond} 
for $ (0,0) \in \conv f(S) $. 
Then 
\begin{align*}
0 & \leq \bsz{\left( \frac{a}{2} , 0 \right) - F(0,0)}{ 
f \left( \frac{a}{2} , 0 \right) - (0,0)} 
 = 
\bsz{\left( \frac{a}{2} , 0 \right) - (a,b)}{ 
\left( \frac{a^2}{4} , 0 \right) - (0,0)} 
\\ & = 
\bsz{ \left( - \frac{a}{2}, -b \right) }{
\left( \frac{a^2}{4}, 0 \right)} = - \frac{a^3}{8} < 0, 
\end{align*}
a contradiction. Thus $ f $ cannot have an 
extended monotone left-inverse. 
\end{proof}
In \cite[5. Remark]{60ISFE} it was proved that for 
the function $ f : S \map \RR^2 $ in the previous 
example there is no continuous extension of 
$ f^{-1} $ to $ \conv f(S) $. 
This is an interesting difference compared to the 
one-dimensional setting, where the ordinary inverse 
always has a continuous {\em and} monotone 
extension to the convex hull of the image set. 

However, it is possible to show that in the 
case of a closed, convex domain the 
generalized inverse function is continuous. 
Before we prove that, we formulate an easy 
geometric observation about euclidean spaces.  

\begin{prop}\label{prop-AngleLimit}
Let $ (x_k) : \NN \map \Rn $ be a sequence, 
$ h \in \Rn $ be a vector and $ \delta > 0 $ be a number 
such that $ \nma h \nma = 1 $, 
$ \nma x_k \nma > \delta $ for all $ k \in \NN $, moreover 
\[
\alpha_k := 
\bsz	{\frac{x_k}{\nma x_k \nma}}{h} \to 1 
\ \mbox{ as } \ k \to \infty. 
\]
Then, for any $ d \in \left( 0, \delta \right) $, 
we have 
\[
\lim_{k \to \infty} 
\bsz{\frac{x_k-dh}{\nma x_k - dh \nma}}{h} 
= 1.
\] 
\end{prop}

\begin{proof}
Consider the function $ g(t) := t-2d+\frac{d^2}{t} = 
\frac{1}{t} \left( t-d \right)^2 $ 
defined for all $ t > 0 $. 
Clearly $ g $ is continuous, 
non-negative and 
$ \lim_{t \to \infty} g(t) = + \infty $. 
Since $ \delta > d $, there exists $ M > 0 $ such that 
$ g(t) > M $ for all $ t \geq \delta $. 
In particular, $ g \left( \nma x_k \nma \right) > M $ 
for every index $ k \in \NN $. 
 
Since $ \alpha_k \leq 1 $ due to the CBS inequality, 
we have 
\[ 
\nma x_k \nma - 2d \alpha_k + \frac{d^2}{\nma x_k \nma} \geq 
\nma x_k \nma - 2d + \frac{d^2}{\nma x_k \nma} = 
g \left( \nma x_k \nma \right) > M 
\hspace{10mm} \left( k \in \NN \right). 
\] 
Now we shall calculate 
\begin{align*}
\left( \frac{\nma x_k \nma - d}{\nma x_k - d h \nma} 
\right)^2 
& = 
\frac{\nma x_k \nma ^2 - 2d \nma x_k \nma + d^2}{
\nma x_k \nma ^2 - 2d \bsz{x_k}{h} + d^2} 
= 
1 + 2d \frac{\bsz{x_k}{h}-\nma x_k \nma}{
\nma x_k \nma ^2 - 2d \bsz{x_k}{h} + d^2} \\ 
& = 
1 + 2d \frac{\alpha_k - 1}{
\nma x_k \nma - 2d \alpha_k + \frac{d^2}{\nma x_k \nma}} 
\to 1 
\ \mbox{ as } k \to \infty, 
\end{align*}
because the numerator tends to $ 0 $ while the 
denominator is bounded from below by $ M > 0 $. 
Consequently, the positive sequence 
$ \frac{\nma x_k \nma - d}{\nma x_k - d h \nma} $ 
also tends to $ 1 $. Finally, 
\begin{align*}
1 & \geq \bsz{\frac{x_k-dh}{\nma x_k - dh \nma}}{h} = 
\frac{\nma x_k \nma}{\nma x_k - dh \nma} \cdot 
\bsz	{\frac{x_k}{\nma x_k \nma}}{h} - 
\frac{d}{\nma x_k - dh \nma} = 
\frac{\nma x_k \nma - d}{\nma x_k - d h \nma} \cdot \alpha_k 
\\ & + \frac{d}{\nma x_k - dh \nma} \left( \alpha_k - 1 \right) 
> 
\frac{\nma x_k \nma - d}{\nma x_k - d h \nma} \cdot \alpha_k 
+ \frac{d}{\delta-d} \left( \alpha_k - 1 \right) 
\to 1 \cdot 1 + \frac{d}{\delta-d} \cdot 0 = 1. 
\end{align*}
Therefore 
$ \bsz{\frac{x_k-dh}{\nma x_k - dh \nma}}{h} \to 1 $ 
and the proof is complete. 
\end{proof}

\begin{thm}\label{thm-ExmoContinuous}
Let $ \emptyset \neq C \subseteq \Rn $ be a closed, 
convex set and let $ f : C \map \Rn $ be 
a strictly increasing mapping. Then the unique 
\exmo function 
$ f^{(-1)} : \conv \left( f(C) \right) \map C $ is 
continuous. 
\end{thm}

\begin{proof}
We have to show that if a sequence 
$ \left( u_k \right) : \NN \map \conv (f(C)) $ 
converges to $ v \in \conv (f(C)) $ then 
$ f^{(-1)}(u_k) $ converges to $ f^{(-1)} (v) $. 

Suppose that, on the contrary, 
$ \lim_{k \to \infty} f^{(-1)}(u_k) \neq f (v) $. 
This means that there exists $ \delta > 0 $ 
and a subsequence $ \left( u_{\ell_k} \right) $ of 
$ (u_k) $ such that 
$ \nma f^{(-1)} \left( u_{\ell_k} \right) 
- f^{(-1)}  \left( v \right) \nma > \delta $. 
Without loss of generality we may assume that this 
sequence is the whole $ \left( u_k \right) $. 
Let us observe that 
\[
\frac{f^{(-1)}\left( u_k \right) - 
f^{(-1)}\left( v \right)}{\nma f^{(-1)}\left( u_k \right) - 
f^{(-1)}\left( v \right) \nma} 
\in \mathcal{S}_{n-1} 
\ \mbox{ for all } k \in \NN. 
\]
Since $ \mathcal{S}_{n-1} $ is compact, the above sequence 
has a convergent subsequence. 
That is, there exist $ h \in \Rn $ with 
$ \nma h \nma = 1 $, and strictly 
increasing sequence of indices 
$ (m_k) : \NN \map \NN $ such that if 
$ x_k := 
f^{(-1)}\left( u_{m_k} \right) - 
f^{(-1)}\left( v \right) $ 
for all $ k \in \NN $, then 
\[ 
\left\nma \frac{x_k}{\nma x_k \nma} - h \right\nma 
\to 0 
\ \mbox{ or, equivalently, } \ 
\bsz{\frac{x_k}{\nma x_k \nma}}{h} \to 1 . 
\]
As $ C $ is convex, $ f^{(-1)}(v) + 
\frac{\delta}{\nma x_k \nma} x_k  \in C $ for 
each $ k \in \NN $. But $ C $ is closed as well, so 
$ f^{(-1)}(v) + \delta h \in C $, whence 
\[ 
s := f^{(-1)}(v) + \frac{\delta}{2} h \in C 
\] also holds. 
Now let us apply Proposition \ref{prop-AngleLimit} 
to the sequence $ (x_k) $, the vector $ h $ and 
the constant $ \delta $. The Proposition claims that, 
in particular, for $ d=\frac{\delta}{2} $ the assertion 
\begin{equation}\label{eq-thm_Cont1}
1 = \lim_{k \to \infty} 
\bsz{\frac{x_k-\frac{\delta}{2}h}{
\nma x_k - \frac{\delta}{2}h \nma}}{h} 
= 
\lim_{k \to \infty} 
\bsz{\frac{ f^{(-1)}\left( u_{m_k} \right) - s}{
\nma f^{(-1)}\left( u_{m_k} \right) - s \nma}}{h} 
\end{equation} 
is valid. 
In the next step let us fix another vector 
$ t := \frac{f^{(-1)}(v) + s}{2} \in C $. Then 
\begin{align*}
\frac{\delta}{2} \bsz{h}{f(s)-v} & = 
\bsz{s - f^{(-1)}(v)}{f(s)-v} = 
\bsz{s-t}{f(s)-f(t)} + \bsz{s-t}{f(t)-v} \\ 
& + \bsz{t-f^{(-1)}(v)}{f(s)-v} = 
\bsz{s-t}{f(s)-f(t)} \\ 
& + \bsz{t - f^{(-1)}(v)}{f(t)-v} + 
\frac{1}{2} \bsz{s - f^{(-1)}(v)}{f(s)-v} > 0, 
\end{align*} 
because the first term is positive as $ f $ is strictly 
increasing, while the other two terms are 
non-negative, due to the defining properties of 
$ f^{(-1)} $. Thus 
$ \bsz{h}{f(s)-v} > 0 $, which implies that there 
exists $ \varepsilon > 0 $ and $ k_0 \in \NN $ such that 
\[
\bsz{h}{f(s)-u_{m_k}} > \varepsilon 
\ \mbox{ for all } k > k_0 \, 
\]
since $ u_{m_k} \to v $. 
Furthermore, from \eqref{eq-thm_Cont1} we get that  
\[ 
\left \vert 
\bsz	{f(s) - u_{m_k}}{\frac{ f^{(-1)}\left( u_{m_k} \right) - s}{
\nma f^{(-1)}\left( u_{m_k} \right) - s \nma} - h} 
\right \vert 
\leq 
\left \nma f(s) - u_{m_k} \right \nma \cdot 
\left \nma \frac{ f^{(-1)}\left( u_{m_k} \right) - s}{
\nma f^{(-1)}\left( u_{m_k} \right) - s \nma} - h \right \nma 
\to 0, 
\]
applying the CBS inequality and 
using that $ \left( f(s) - u_{m_k} \right) $ 
is a bounded sequence. 
In particular, there exists an index $ k_1 > k_0 $ 
such that 
\[
\bsz	{f(s) - u_{m_k}}{\frac{ f^{(-1)}\left( u_{m_k} \right) - s}{
\nma f^{(-1)}\left( u_{m_k} \right) - s \nma} - h} 
> - \frac{\varepsilon}{2} 
\ \mbox{ for all } k > k_1. 
\]
Combining the previous results we obtain that, 
for all $ k > k_1 \,$, the inequality 
\begin{align*}
& \bsz{f(s)-u_{m_k}}{\frac{ f^{(-1)}\left( u_{m_k} \right) - s}{
\nma f^{(-1)}\left( u_{m_k} \right) - s \nma}} 
= 
\bsz{f(s)-u_{m_k}}{\frac{ f^{(-1)}\left( u_{m_k} \right) - s}{
\nma f^{(-1)}\left( u_{m_k} \right) - s \nma}-h} \\ 
& + \bsz{f(s)-u_{m_k}}{h} > - \frac{\varepsilon}{2} 
+ \varepsilon = \frac{\varepsilon}{2} > 0 
\end{align*} 
holds. 
But this implies 
\[
\bsz{f^{(-1)}\left( u_{m_k} \right) - s}{u_{m_k} - f(s)} 
< 0 
\hspace{10mm} \left( k > k_1 \right), 
\]
which contradicts the monotonicity of $ f^{(-1)} $. 
Therefore our initial assumption about 
$ f^{(-1)}(u_k) $ not converging to $ f^{(-1)}(v) $ 
was false, so $ f^{(-1)} $ is continuous. 
\end{proof}

\section{Application for quasi-arithmetic means}\label{Section-Appl} 

In the final section of our paper we introduce the 
concept of vector valued quasi-arithmetic means. 
\begin{defin}
Let $ \emptyset \neq C \subseteq \Rn $ be a 
closed, convex set and let $ f : C \map \Rn $ 
be a strictly monotone mapping. Moreover, 
let $ m \in \NN $ and $ \lambda_1 \,, \dots , \lambda_m > 0 $ 
be fixed numbers such that $ m \geq 2 $ and 
$ \lambda_1 + \dots + \lambda_m = 1 $. 

Then the function 
$ \mathcal{M}_{f, \lambda_1 , \dots , \lambda_m}^{[m]} 
: C^m \map C $ defined by 
\[
\mathcal{M}_{f, \lambda_1 , \dots , \lambda_m}^{[m]} 
\left( x_1 \,, \dots , x_m \right) 
:= f^{(-1)} \bigl( \lambda_1 f(x_1) + \dots 
+ \lambda_m f(x_m) \bigr) 
\hspace{8mm} 
\left( x_1 \,, \dots , x_m \in C \right) 
\] 
is called the 
{\em $ m $-variable weighted quasi-arithmetic mean}. 
The {\em generator} of the mean is $ f $. 
\end{defin}

It is easy to see that for arbitrary weights 
$ \mathcal{M}_{f, \lambda_1 , \dots , \lambda_m}^{[m]} = 
\mathcal{M}_{-f, \lambda_1 , \dots , \lambda_m}^{[m]} \, $, 
thus it is enough to consider the case of a strictly 
increasing generator. In \cite{Leo23} the idea of 
vector-valued QAMs appears, although it is 
{\em a priori assumed} that the generator has a convex image, 
which we find rather restrictive. 

One of the first natural problems which arises 
immediately is the equality problem: for a 
fixed number of variables and fixed weights find 
those generators for which 
\begin{equation}\label{eq_EqualityProblem}
\mathcal{M}_{f, \lambda_1 , \dots , \lambda_m}^{[m]} 
\left( x_1 \,, \dots , x_m \right) 
= 
\mathcal{M}_{g, \lambda_1 , \dots , \lambda_m}^{[m]} 
\left( x_1 \,, \dots , x_m \right) 
\hspace{8mm} 
\left( x_1 \,, \dots , x_m \in C \right). 
\end{equation}
In the scalar valued setting when the generators 
are supposed the be continuous, 
then the solution is a well-known classical result: 
$ f = a \cdot g + b $ with some real numbers such that 
$ a \neq 0 $ (see \cite[Theorem 83 and Section 3.7]{HLP34} 
or \cite[Theorem A]{PP26}). 

In the case of non-continuous generators 
the situation is much more difficult. Even though it has been 
shown recently in \cite[Corollary 4.2]{PP26} 
that if the equality of 
two QAMs holds 
{\em for all} $ m \in \NN $, then the generators are 
affine transforms of each other, the same cannot 
be claimed in the case when we have 
{\em a fixed number of variables} $ m $. 
What is more, a fresh result of Kiss and Pasteczka 
gives examples for three-variable 
QAMs which are equal but their 
generators are not affine transforms of each other 
(see \cite[Example 1]{KP26}). 

The complete characterization of the solutions of the 
equality problem for fixed weights and 
a fixed number of variables is 
still open even in the real valued case. 
Therefore our aim here, instead of discussing 
the general equality problem, is to solve a particular 
case: for a fixed $ m \geq 2 $ when does the 
$ m $-variable vector-valued weighted quasi-arithmetic 
mean coincide with the weighted arithmetic mean? 
The $ m $-variable arithmetic mean (generated by 
the identity) with weights 
$ \lambda_1 \,, \dots , \lambda_m $ mean will be 
denoted by 
$ \mathcal{A}_{\lambda_1 , \dots , \lambda_m}^{[m]} $.  

\begin{prop}\label{prop-Application_invinj}
Let $ \emptyset \neq C \subseteq \Rn $ be a 
closed, convex set and let $ f : C \map \Rn $ 
be a strictly monotone mapping. Moreover, 
let $ m \geq 2 $ and $ \lambda_1 \,, \dots , \lambda_m > 0 $ 
such that 
$ \lambda_1 + \dots + \lambda_m = 1 $. 
Suppose that 
$ \mathcal{M}_{f, \lambda_1 , \dots , \lambda_m}^{[m]} 
= \mathcal{A}_{\lambda_1 , \dots , \lambda_m}^{[m]} $. 
Then, for all $ x \in C^{\circ} $, 
\[
y \in \conv (f(C)) \setminus \{ f(x) \} 
\ \mbox{ implies } \ 
f^{(-1)}(y) \neq x .
\]
\end{prop}

\begin{proof} 
Firstly, $ \mathcal{M}_{f, \lambda_1 , \dots , \lambda_m}^{[m]} 
= \mathcal{A}_{\lambda_1 , \dots , \lambda_m}^{[m]} $ 
implies that, in particular, 
\begin{equation}\label{eq-prop_Appl_2mean}
\begin{aligned}
f^{(-1)} & \left( \lambda_1 f(x) + 
\left( 1-\lambda_1 \right) f(y) \right) = 
\mathcal{M}_{f, \lambda_1 , \dots , \lambda_m}^{[m]} 
\left( x,y, \dots , y \right) \\ 
& = \mathcal{A}_{\lambda_1 , \dots , \lambda_m}^{[m]} 
\left( x,y, \dots , y \right) = 
\lambda_1 x + \left( 1- \lambda_1 \right) y 
\hspace{15mm} 
\left( x,y \in C \right). 
\end{aligned} 
\end{equation}

Suppose that, contrary to the statement, 
there exists $ x \in C^{\circ} $ and 
$ y \in \conv(f(C)) $ such 
that $ f^{(-1)}(y) = x $ but $ y \neq f(x) $. 
Let us define the unit vector 
$ h := \frac{y-f(x)}{\nma y-f(x) \nma} $. 
Since $ x \in C^{\circ} $, there exists 
$ \alpha_0 > 0 $ such that 
$ x + \alpha h \in C $ for all 
$ \alpha \in (0,\alpha_0) $. Using the monotonicity 
of $ f^{(-1)} $ we obtain 
\begin{align*}
0 \leq  
\bsz	{f(x + \alpha h) - y}{ 
f^{(-1)} \left( f(x + \alpha h) \right) - f^{(-1)}(y) } 
= \alpha \cdot \bsz{f(x + \alpha h) - y}{h}, 
\end{align*}
so $ 0 \leq \bsz{f(x + \alpha h) - y}{h} $. 

Now let us observe that, 
for all $ \alpha \in (0,\alpha_0) $, 
there uniquely exist 
$ \mu_{\alpha} > 0 $ and $ w_{\alpha} \in \Rn $ 
such that $ \mu_{\alpha} \geq \nma y - f(x) \nma $ 
and $ \bsz{w_{\alpha}}{h} = 0 $, moreover 
\begin{equation}\label{eq-prop_Appl-decomp}
f(x + \alpha h) = f(x) + \mu_{\alpha} h + w_{\alpha}. 
\end{equation}
%
Indeed, 
the decomposition \eqref{eq-prop_Appl-decomp} is 
due to the Gram--Schmidt algorithm with 
\[
\mu_{\alpha} = \bsz{f(x + \alpha h)-f(x)}{h} 
\hspace{3mm} \mbox{ and } \hspace{3mm} 
w_{\alpha} = f(x+ \alpha h) - f(x) - \mu_{\alpha} h, 
\]
while the previous step of the proof yields 
\begin{align*}
0 & \leq \bsz{f(x + \alpha h) - y}{h} = 
\bsz{f(x + \alpha h) - f(x) + f(x) -y}{h} = 
\bsz{f(x+\alpha h)-f(x)}{h} \\ 
& + \bsz{f(x) - y}{h} 
= \mu_{\alpha} - \bsz{y-f(x)}{\frac{y-f(x)}{\nma y-f(x) \nma}} 
= \mu_{\alpha} - \nma y - f(x) \nma, 
\end{align*}
which is equivalent to 
$ \mu_{\alpha} \geq \nma y - f(x) \nma $. 
Therefore 
\begin{equation}\label{eq-prop_Appl_nu}
\nu := \inf 
\left \lbrace 
\bsz{f(x+\alpha h)-f(x)}{h} \ : \ 
\alpha \in \left( 0,\alpha_0 \right) 
\right \rbrace
\geq 
\nma y - f(x) \nma . 
\end{equation}

In the next step let us choose 
$ \beta \in \left( 0 , \alpha_0 \right) $ such that 
for $ z := x + \beta h \in C $ we have 
\begin{equation}\label{eq-prop_Appl-zchoice}
\nu \leq \bsz{f(z)-f(x)}{h} < 
\frac{\nu}{\left( 1-\lambda_1 \right)}. 
\end{equation} 
Moreover, let us fix 
$ u := \frac{1+\lambda_1}{2} x + \frac{1-\lambda_1}{2} z 
= x + \frac{1-\lambda_1}{2} \beta h 
\in C $. Then 
\begin{align*}
0 & \leq 
\bsz{f^{(-1)} 
\bigl( \lambda_1 f(x) + (1-\lambda_1) f(z) \bigr) - 
f^{(-1)} \bigl( f(u) \bigr)}{ 
\bigl( \lambda_1 f(x) + (1-\lambda_1) f(z) \bigr) - f(u)} 
\\ 
& = 
\bsz{ \bigl( \lambda_1 x + (1-\lambda_1) z \bigr) - u}{ 
(1 - \lambda_1)f(z)-(1 - \lambda_1)f(x) + f(x) - f(u)} 
\\ 
& = 
\bsz{\frac{1-\lambda_1}{2} \beta \cdot h}{
\bigl( (1 - \lambda_1)(f(z)-f(x)) \bigr)
- \bigl( f(u) - f(x) \bigr)} 
\\ 
& = 
\frac{1-\lambda_1}{2} \beta \cdot 
\Bigl( (1-\lambda_1) 
\bsz{f(z)-f(x)}{h} - \bsz{f(u)-f(x)}{h} \Bigr) 
\\ 
& < 
\frac{1-\lambda_1}{2} \beta \cdot 
\Bigl( (1-\lambda_1) \frac{\nu}{1-\lambda_1} - \nu \Bigr) 
= 0, 
\end{align*}
where we have used \eqref{eq-prop_Appl_2mean}, 
\eqref{eq-prop_Appl_nu} and \eqref{eq-prop_Appl-zchoice} 
together with the monotonicity of $ f^{(-1)} $. 
However, the obtained contradiction means that 
$ f^{(-1)}(y) = x $ implies $ y = f(x) $, so the 
proof is complete by contraposition. 
\end{proof}

\begin{thm}\label{thm_Appl_vvQAMs}
Let $ m \geq 2 $, and let 
$ f : \Rn \map \Rn $ be a strictly increasing mapping. 
Moreover 
let $ m \geq 2 $ and 
$ \lambda_1 \,, \dots , \lambda_m > 0 $ 
such that 
$ \lambda_1 + \dots + \lambda_m = 1 $. 
Then the following two statements are equivalent: 
\begin{enumerate}
\item[\emph{(i)}] 
For all 
$ x_1 \,, \dots , x_m \in \Rn $ 
it holds that 
\[
\mathcal{M}_{f, \lambda_1 , \dots , \lambda_m}^{[m]} 
\left( x_1 \,, \dots , x_m \right)
= \mathcal{A}_{\lambda_1 , \dots , \lambda_m}^{[m]} 
\left( x_1 \,, \dots , x_m \right)
\] 
\item[\emph{(ii)}]
There exist a positive definite matrix 
$ A \in \mathcal{M}_{n \times n}(\RR) $ and a vector 
$ b \in \Rn $ such that 
\[
f(t) = At + b 
\hspace{8mm} 
\left( t \in \Rn \right). 
\]
\end{enumerate}
\end{thm}

\begin{proof} 
We begin with the proof of (i) $ \Longrightarrow $ (ii). 
From Proposition \ref{prop-Application_invinj} 
we have that $ f^{(-1)} $ is injective, so 
it coincides with the ordinary inverse. 
In particular, 
$ \conv \left( f \left( \Rn \right) \right) = 
f \left( \Rn \right) $, i.~e. $ f \left( \Rn \right) $ 
is convex. Suppose that $ t \in f \left( \Rn \right) $ 
is a boundary point of $ f \left( \Rn \right) $. 
Then there exist a supporting 
hyperplane for $ f \left( \Rn \right) $ at the 
point $ t $ (cf. \cite[Theorem 11.6]{Roc70}). 
This implies the existence of a vector 
$ 0 \neq h \in \Rn $ such that 
\[ 
\bsz{f(x)-t}{h} \geq 0 
\hspace{10mm} \left( x \in \Rn \right). 
\]
However, $ f $ is strictly increasing, so 
for $ x_0 := f^{-1}(t) $ we have 
\[ 
0< \bsz{f(x_0 - h)-f(x_0)}{
(x_0-h)-x_0} = 
\bsz{f(x_0 - h) - t}{-h}, 
\] 
which contradicts the choice of $ h $. Thus 
$ \partial \bigl( f \left( \Rn \right) \bigr) 
\cap f \left( \Rn \right) 
= \emptyset $, so $ f(\Rn) $ is open. 
Hence we may apply Brouwer's invariance of domain 
theorem for the injective, continuous 
(see Theorem \ref{thm-ExmoContinuous}) map 
$ f^{-1} : f \left( \Rn \right) \map \Rn $ and 
conclude that its inverse $ f : \Rn \to f(\Rn) $ 
is also continuous. 

Recalling 
\eqref{eq-prop_Appl_2mean} from the previous proof and 
applying $ f $ to both sides, we have 
\begin{equation}\label{eq-thm_Appl_lambda1}
\lambda_1 f(x) + (1-\lambda_1) f(y) 
= f \left( \lambda_1 x + (1-\lambda_1) y \right) 
\hspace{10mm} \left( x,y \in \Rn \right). 
\end{equation}
Consequently, every coordinate function 
$ f_i : \Rn \map \RR $ (for $ i = 1 , \dots , n $) of $ f $ 
is affine. That is, 
there exist $ a_i \in \Rn $ and $ b_i \in \RR $ 
such that 
\[
f_i(x) = \bsz{a_i}{x} + b_i 
\hspace{10mm} \left( i = 1, \dots , n \right)
\]
for all $ x \in \Rn $. The proof is similar to 
the solution process of Jensen's equation 
(which corresponds to \eqref{eq-thm_Appl_lambda1} with 
$ \lambda_1 = \frac{1}{2} $) discussed in 
\cite[Section 13.2]{Kuc09}. 
A slightly stronger statement can be found in 
\cite[Theorem 3.1]{BT26}. 
Summarizing the previous results, let 
$ A \in \mathcal{M}_{n \times n}(\RR) $ be the matrix 
which contains the vector $ a_i $ in its $ i $-th 
row and let the $ i $-th element of $ b \in \Rn $ 
be $ b_i $ (for $ i = 1 , \dots , n $). 
Then 
\[
f(x) = \left( f_1(x) , \dots , f_n(x) \right) = 
Ax + b 
\hspace{10mm} (x \in \Rn). 
\]
Furthermore, $ A $ has to be positive definite, because, 
for every $ 0 \neq x \in \Rn $, we have 
\[
\bsz{Ax}{x} = \bsz{(Ax+b)-(A \cdot 0+b)}{x-0} = 
\bsz{f(x)-f(0)}{x-0} > 0. 
\]

Finally, for the converse direction 
(ii) $ \Longrightarrow $ (i), 
let $ A $ be positive definite, $ b $ be arbitrary, 
and define $ f(t) = At+b $ for all $ t \in \Rn $. Then 
$ f^{-1}(y) = A^{-1}(y-b) $ for all $ y \in \Rn $, thus 
\begin{align*}
\mathcal{M}_{f, \lambda_1 , \dots , \lambda_m}^{[m]} 
& \left( x_1 \,, \dots , x_m \right)
= A^{-1} \bigl( \lambda_1 \left( A(x_1) + b \right) + 
\dots \lambda_m \left( A(x_m) + b \right) - b \bigr) \\  
& = 
A^{-1} \bigl( A \left( 
\lambda_1 x_1 + \dots + \lambda_m x_m \right) \bigr) = 
\mathcal{A}_{\lambda_1 , \dots , \lambda_m}^{[m]} 
\left( x_1 \,, \dots , x_m \right) 
\end{align*}
\end{proof}

\section{Concluding remarks and open problems} 

In Theorem \ref{thm_Appl_vvQAMs} the domain of the 
QAM is the whole space $ \Rn $. Although this is quite 
restrictive, we demonstrate that $ \Rn $ 
cannot be replaced by an arbitrary closed convex set. 

\begin{exmp}
Let 
$ C := \lbrace (t,0) \in \RR^2 : t \in [0,1] \rbrace $ 
and define $ f : C \map \Rn $ by 
\[
f \left( t,0 \right) := \left( t, g(t) \right) 
\hspace{10mm} 
\left( \, t \in [0,1] \, \right), 
\]
where $ g : [0,1] \map \RR $ is an arbitrary function. 
It is easy to check that $ f $ is strictly 
increasing, moreover the generalized left inverse is 
\[
f^{(-1)} \left( t,s \right) = (t,0) 
\hspace{10mm} 
\bigl( \, (t,s) \in \conv(f(C)) \, \bigr). 
\]
Clearly, $ f $ is not necessarily an affine function. 
On the other hand, for arbitrary vectors 
$ (x,0), (y,0) \in C $ we have 
\begin{align*}
& \mathcal{M}_{f,\frac{1}{2}, \frac{1}{2}}^{[2]} 
\left( (x,0),(y,0) \right) =  
f^{(-1)} \left( \frac{f(x,0)+f(y,0)}{2} \right) =  
f^{(-1)} \left( \frac{(x,g(x))+(y,g(y))}{2} \right) \\ 
& = 
f^{(-1)} \left( \frac{x+y}{2} , \frac{g(x)+g(y)}{2} \right) 
= \left( \frac{x+y}{2} , 0 \right) = 
\frac{(x,0) + (y,0)}{2} = 
\mathcal{A}_{\frac{1}{2}, \frac{1}{2}}^{[2]} 
\left( (x,0),(y,0) \right).
\end{align*}
\end{exmp}

On the other hand, this counterexample involves a 
pathological domain with empty interior. 

\begin{opprob}
Is it possible to prove the analogue of 
Theorem \ref{thm_Appl_vvQAMs} for a restricted domain, 
namely for any closed convex set $ C \subset \Rn $ 
such that $ C^{\circ} \neq \emptyset $?  
\end{opprob}

\begin{opprob} 
Is it possible to solve the equality problem of 
QAMs assuming that the generators are {\em continuous}? 
That is, how can we 
characterize the strictly increasing generators 
$ f, g : C \map \Rn $ for which \eqref{eq_EqualityProblem} 
holds on an appropriate domain $ C \subseteq \Rn $? 
\end{opprob}

\begin{opprob}
We shall note that Definition \ref{def-StrictlyMon} 
makes sense for an 
arbitrary inner product space instead of $ \Rn $. Even though 
throughout the paper we use methods which are typically 
finite dimensional, it is natural to ask the following: 
Does there exist an extended monotone 
left-inverse for a strictly monotone mapping 
$ f : C \map H $, where $ C $ is a closed, convex subset 
of an infinite dimensional Hilbert space $ H $? 
\end{opprob}



\end{document}